\documentclass[12pt]{amsart}

\usepackage{subcaption}
\usepackage{amssymb,latexsym,amsmath,amsthm,graphicx}
\usepackage{cite}

\newtheorem{theorem}{Theorem}
\newtheorem{lemma}[theorem]{Lemma}

\newtheorem*{conjecture}{Conjecture}
\theoremstyle{remark}

\begin{document}
\author[A. Thomack, Z. Tyree]{Andrew Thomack and Zachariah Tyree}
\title[Random harmonic polynomials]{On the zeros of random harmonic polynomials: the Weyl model}

\begin{abstract}
Li and Wei (2009) studied the density of zeros of Gaussian harmonic polynomials with independent Gaussian coefficients. They derived a formula for the expected number of zeros of random harmonic polynomials as well as asymptotics for the case that the polynomials are drawn from the Kostlan ensemble. In this paper we extend their work to cover the case that the polynomials are drawn from the Weyl ensemble by deriving asymptotics for this class of harmonic polynomials.
\end{abstract}

\maketitle

\section{Introduction}
A harmonic polynomial is a complex-valued harmonic function given by:
\[H_{n,m}(z)=p(z)+q(\overline{z}),\]
where $p$ and $q$ are analytic complex polynomials of degree $n$ and $m$ respectively and $\overline{z}$ denotes the complex conjugate of $z$.  Since $H$ is not analytic, the Fundamental Theorem of Algebra does not apply, and it is natural to ask how many zeros $H$ can have, i.e. how many solutions are there to $H_{n,m}(z)=0$.

Assuming $n\ge m$ and denoting by $\mathcal{N}(T)$ the number of zeros of $H$ in a domain $T\subseteq\mathbb{C}$ we can bound $\mathcal{N}$  for a generic choice of $p$ and $q$ as such:
\[n\le\mathcal{N}(\mathbb{C})\le n^2.\]
The upper bound is due to Wilmshurst who applied Bezout's Theorem to the real and imaginary part of $H(z)=0$ \cite{Wilmshurst}.  The lower bound is a consequence of a generalized argument principle.  In fact, these bounds are sharp for each $n$, though for $m=1$ the upper bound has been improved to $3n-2$ \cite{Khavinson}, and it has been conjectured for fixed $m$ that the upper bound is linear in $n$ \cite{Saez}, \cite{Lee}, \cite{Wilmshurst}.

Given the wide range of values that $\mathcal{N}$ can take and the lack of an explicit formula in $n$ and $m$ for $\mathcal{N}$, the next question is: given an ``arbitrary'' harmonic polynomial, what is the expected value of $\mathcal{N}$?

The question has been well studied in the framework of analytic polynomials.  Kac initiated the study by deriving an explicit formula for the expected number of zeros of a random real polynomial \cite{Kac}.  Subsequently, other authors developed similar formulae for trigonometric polynomials \cite{Dunnage}, complex polynomials over an arbitrary domain \cite{Vanderbei}, and polynomial vector fields \cite{Azais}.

In the context of harmonic polynomials, Li and Wei showed an explicit formula for $\mathbb{E}[\mathcal{N}_{n,m}]$ when the coefficients are independent complex Gaussians \cite{LiWei}.  Moreover, they showed that if $H$ satisfies
\begin{equation}
	\label{int:harmonic}
	H_{n,m}(z)=\sum_{j=0}^na_jz^j+\sum_{j=0}^mb_j\overline{z}^j
\end{equation}	
with $0\le m\le n$, where $a_j$ and $b_j$ are complex Gaussian random variables satisfying:
\begin{equation}\label{int:exp}
	\mathbb{E}[a_j]=\mathbb{E}[b_j]=0\text{ , }
	\mathbb{E}[a_j\overline{a}_k]=\delta_{jk}\binom{n}{j}\text{ , and }
	\mathbb{E}[b_j\overline{b}_k]=\delta_{jk}\binom{m}{j}
\end{equation}
then as $n\rightarrow\infty$,
\begin{equation*}
	\mathbb{E}[\mathcal{N}_{n,m}]\sim\left\{
	\begin{array}{ll}
		\frac{\pi}{4}n^{3/2}&m=n
		\\n&m=\alpha n+o(n),\ \alpha\in[0,1).
	\end{array}\right.
\end{equation*}

Notice that in the case where $m=\alpha n$ the modulus of $p$ is much larger than that of $q$ so that $H$ tends toward an analytic polynomial as $n$ increases.  Similarly, in this case $H$ asymptotically obeys the fundamental theorem of algebra.

A related result was proved by Lerario and Lundberg.  Choosing a slightly different definition of ``random,'' Lerario and Lundberg showed that if in (\ref{int:exp}) one instead defines $\mathbb{E}[b_j\overline{b}_k]=\delta_{jk}\binom{j}{k}$ then
\[\mathbb{E}[\mathcal{N}_{n,m}]\sim c_\alpha n^{3/2}\qquad\text{when $m=\alpha n$}\]
where $c_\alpha$ is a constant depending only on $\alpha\in(0,1]$ \cite{Lundberg}.  Moreover, $c_\alpha\rightarrow\frac{\pi}{4}$ as $\alpha\rightarrow1$ giving the asymptotic value of $\mathbb{E}[\mathcal{N}_{n,m}]$ a satisfying continuity in $\alpha$, a property shared by the particular model of random $H$ that is the focus of this article.

A well studied choice  of random coefficients for polynomilas is for $a_j$ and $b_j$ to be i.i.d. complex Gaussians (e.g. $\mathbb{E}[a_j\overline{a}_k]=\mathbb{E}[b_j\overline{b}_k]=\delta_{jk}$).  The first author provided asymptotics for the expected number of zeros in the cases when $m$ is fixed and when $m = n$ \cite{Thomack}.
\[\mathbb{E}[\mathcal{N}_{n,m}]\sim n\text{ as $n\rightarrow\infty$ and $m$ is fixed,}\] 
and there exists $c_1,c_2>0$ such that for large $n$
\[c_1n\log n\le\mathbb{E}[\mathcal{N}_{n,n}]\le c_2n\log n.\]

In this paper, we introduce another Gaussian model of harmonic polynomials, obtained by independently sampling $p$ and $q$ from the Weyl model.  That is, the polynomials satisfying (\ref{int:harmonic}) where $a_j$ and $b_j$ are complex Gaussian random variables satisfying:
\begin{equation}
	\mathbb{E}[a_j]=\mathbb{E}[b_j]=0\text{ , }\mathbb{E}[a_j\overline{a}_k]
	=\delta_{jk}\frac{1}{j!}\text{ , and }\mathbb{E}[b_j\overline{b}_k]
	=\delta_{jk}\frac{1}{j!}.
\end{equation}
The number of zeros when $m$ is fixed resembles that of the analogous complex analytic Weyl polynomials studied in \cite{Tao}.  The real Weyl polynomials have been found to have an expected number of real zeros asymptotic to $\sqrt{n}$ \cite{Kostlan}.
The formula for the expected number of zeros obtained by Li and Wei can be extended to more general Gaussian models of harmonic polynomials (see Theorem  \ref{thmLiWeiGeneral} below) including the Weyl model.  Using this result along with classical methods in the asymptotic analysis of integrals, we prove the following:
\begin{theorem}\label{MainResult}
	If $\mathcal{N}_{n,m}$ denotes the number of zeros of a random Weyl polynomial then
	\begin{equation}
		\label{MainThm}
		\mathbb{E}\left[\mathcal{N}_{n,m}(\mathbb{C})\right]\sim\left\{
		\begin{array}{ll}
			n&m\text{ is fixed,}
			\\\frac{1}{3}m^{3/2}&m=\alpha n+O(1),\ \alpha\in(0,1]
		\end{array}\right._.
	\end{equation}
\end{theorem}
Moreover, the proof of the theorem provides detailed information on the so-called ``first intensity'', the average local density of zeros which is determined by the integrand appearing in the Kac-Rice formula stated in Theorem \ref{thmLiWeiGeneral}.  We observe distinct behavior over three regions (see Figure \ref{denistyPlots}).  The high density in the central core is particularly striking and provokes further study.  It would be desireable to find a heuristic explanation of this phenomenon.
\begin{figure}[h]
	\begin{subfigure}[b]{.35\textwidth}
		\includegraphics[width=\textwidth]{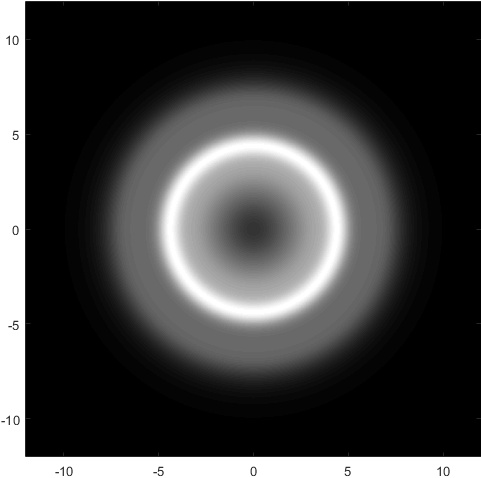}
		\caption{$m=16$, $n=64$}
	\end{subfigure}
	\begin{subfigure}[b]{.35\textwidth}
		\includegraphics[width=\textwidth]{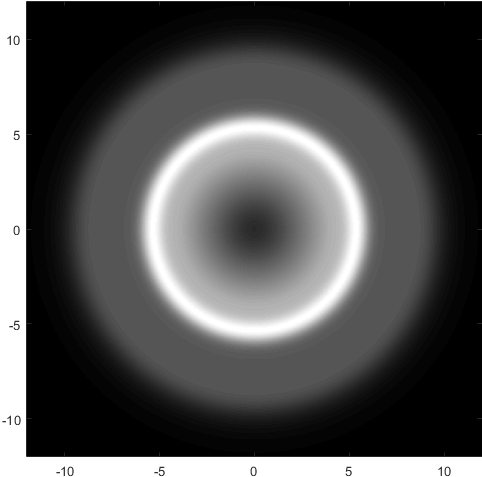}
		\caption{$m=25$, $n=100$}
	\end{subfigure}
	\caption{The first intensity function as a density plot for two values of $n$ 	and $m=0.25n$.  In both plots we see three regions, one disc of radius $\sqrt{m}$ containing a high density of zeros, one annulus of radius $\sqrt{n}$ with a less dense, almost constant, distribution of zeros, and the complement of the disc of radius $\sqrt{n}$ which has extremely low density of zeros.}
	\label{denistyPlots}
\end{figure}

\section{A Kac-Rice formula for Gaussian harmonic polynomials}
\begin{theorem}\label{thmLiWeiGeneral}
	The expectation $\mathbb{E}\mathcal{N}_H(T)$ of the number of zeros of 
	\[H_{n,m}(z)=\sum_{j=0}^na_jz^j+\sum_{j=0}^mb_j\overline{z}^j\]
	with $\mathbb{E}a_j=\mathbb{E}b_j=0$ and 
	$\mathbb{E}a_j\overline{a}_j=\alpha_j$ and 
	$\mathbb{E}b_j\overline{b}_j=\beta_j$ on a domain 
	$T\subset\mathbb{C}$ is given by:
	\begin{equation}
		\label{LiWeiGeneral}
		\mathbb{E}\mathcal{N}_F(T)=\frac{1}{\pi}\int_T\frac{1}{\lvert z\rvert^2}
		\frac{r_1^2+r_2^2-2r_{12}^2}{r_3^2\sqrt{(r_1+r_2)^2-4r_{12}^2}}dA(z),
	\end{equation}
	where $dA(z)$ denotes the Lebesgue measure on the plane, and
	\begin{align*}
		r_3=&\sum_{j=0}^n\alpha_j\lvert z\rvert^{2j}
		+\sum_{j=0}^m\beta_j\lvert z\rvert^{2j},\qquad\qquad\quad 
		r_{12}=\left(\sum_{j=1}^nj\alpha_j\lvert z\rvert^{2j}\right)
		\left(\sum_{j=1}^mj\beta_j\lvert z\rvert^{2j}\right),
		\\r_1=&r_3\sum_{j=1}^nj^2\alpha_j\lvert z\rvert^{2j}
		-\left(\sum_{j=1}^nj\alpha_j\lvert z\rvert^{2j}\right)^2,\ \ 
		r_2=r_3\sum_{j=1}^mj^2\beta_j\lvert z\rvert^{2j}
		-\left(\sum_{j=1}^mj\beta_j\lvert z\rvert^{2j}\right)^2.
	\end{align*}
\end{theorem}
	As written, this theorem is more general than as it is written in \cite{LiWei} and \cite{Lundberg}, by replacing specific variances of $a_j$ and $b_j$ with arbitrary positive values $\alpha_j$ and $\beta_j$, though the same proof suffices.

\section{Proof of Theorem \ref{MainResult}}
Our strategy for finding the asymptotic value of $\mathcal{N}_H$ as $n\rightarrow\infty$ is to find an appropriate change of variables, and then use the Lebesgue Dominated Convergence Theorem, or the following generalization of it from \cite[\S4.4]{Royden}:
\begin{theorem}[General Lebesgue Dominated Convergence Theorem]\label{GLDCT}
Let $\{f_n\}$ be a sequence of measurable functions on $E$ that converges p.w. a.e. on $E$ to $f$.  Suppose there is a sequence $\{g_n\}$ of nonnegative measurable functions on $E$ that converges p.w. a.e. on $E$ to $g$ and dominates $\{f_n\}$ on $E$ in the sense that $\lvert f_n\rvert\le g_n$ on $E$ for all $n$.  If 
\[\lim_{n\rightarrow\infty}\int_Eg_n=\int_Eg<\infty,\ \ then\ \ \lim_{n\rightarrow\infty}\int_Ef_n=\int_Ef.\]
\end{theorem}
We proceed by finding the limit of the integrand, then we show that we can find a sequence of integrable bounding functions satisfying the theorem above.

We first apply Theorem \ref{LiWeiGeneral} and introduce some useful notation.  Let
\[a_n=\sum_{j=0}^n\frac{\lvert z\rvert^{2j}}{j!},\quad b_n=\sum_{j=1}^nj\frac{\lvert z\rvert^{2j}}{j!},\quad c_n=\sum_{j=1}^nj^2\frac{\lvert z\rvert^{2j}}{j!}.\]
Then we may write $r_1$, $r_2$, $r_3$, and $r_{12}$ as follows
\begin{equation*}
r_3=a_n+a_m,\quad r_{12}=b_nb_m,\quad r_1=r_3c_n-b_n^2,\quad r_2=r_3c_m-b_m^2.
\end{equation*}
Throughout the proof we may apply change of variables, in which case it is assumed that $a_n$, $b_n$, and $c_n$ are the same function with the appropriate change of variables applied.
\subsection{Pointwise limit of the first intensities}
We begin by addressing the case when $m=\alpha n$, $\alpha\in(0,1]$.  We first notice that we have some convenient relationships between $a_k$, $b_k$, and $c_k$.
\begin{equation}
b_k=\lvert z\rvert^2a_{k-1}\qquad c_k=\lvert z\rvert^4a_{k-2}+\lvert z\rvert^2a_{k-1}
\end{equation}Since the limit of $a_n$ is equal to that of $a_{n-1}$, $a_{n-2}$, $a_m$, $a_{m-1}$, and $a_{m-2}$ we get that as $n\rightarrow\infty$
\begin{equation}
\frac{r_1^2+r_2^2-2r_{12}^2}{\lvert z\rvert^2r_3^2\sqrt{(r_1+r_2)^2-4r_{12}^2}}\rightarrow\frac{\sqrt{\lvert z\rvert^2+1}}{2}.
\end{equation}
However, this pointwise convergence is not dominated by an integrable function.  In order to obtain a dominated convergence we will need to perform a change of variable and divide by the constant $n^{3/2}$ which allows us to integrate the limit of the new integrand.

We perform the change of variables, first from Cartesian to polar,
\begin{equation*}
\frac{1}{\pi}\int_{\mathbb{C}}\frac{r_1^2+r_2^2-2r_{12}^2}{\lvert z\rvert^2r_3^2\sqrt{(r_1+r_2)^2-4r_{12}^2}}dA(z)=2\int_0^\infty\frac{r_1^2+r_2^2-2r_{12}^2}{\lvert z\rvert r_3^2\sqrt{(r_1+r_2)^2-4r_{12}^2}}d\lvert z\rvert,
\end{equation*}
and then using the change of variables $\lvert z\rvert^2=nt$ this becomes
\begin{equation}\label{MainInt}
\int_0^\infty\frac{r_1^2+r_2^2-2r_{12}^2}{t r_3^2\sqrt{(r_1+r_2)^2-4r_{12}^2}}dt,
\end{equation}
where $r_1$, $r_2$, $r_3$ and $r_{12}$ are defined as before with $nt$ substituted for $\lvert z\rvert^2$.  Note that with this change of variable, the notation $a_n$ becomes ambiguous, thus we clarify that the $n$ in $a_n$ refers only to the upper bound in the index of the sum, that is, $a_{n-1}=\sum_{j=0}^{n-1}(nt)^j/j!$.

We use the incomplete Gamma function in combination with Laplace's method to evaluate the limit of the integrand.  We have the following identity for the incomplete Gamma function when $n\in\mathbb{N}$,
\[\Gamma(n,t):=\int_t^\infty x^{n-1}e^{-x}dx=(n-1)!e^{-t}\sum_{j=0}^{n-1}\frac{t^j}{j!}\]
which implies
\[\sum_{j=0}^n\frac{t^j}{j!}=\frac{e^t\Gamma(n+1,t)}{n!}\]
and for a function $k=k(n)$ which is $O(1)$ as $n\rightarrow \infty$ such that $\alpha n +k\in\mathbb{N}$,
\begin{equation}
	\label{incGamma}
	\sum_{j=0}^{\alpha n-k}\frac{(nt)^j}{j!}
	=\frac{e^{nt}\Gamma(\alpha n-k+1,nt)}{(\alpha n-k)!}
	=\frac{e^{nt}}{(\alpha n-k)!}\int_{nt}^\infty x^{\alpha n-k}e^{-x}dx.
\end{equation}
We make the change of variables, $x=ny$
\begin{equation}
	\label{GammaChange}
	\int_t^\infty(ny)^{\alpha n-k}e^{-ny}n\ dy
	=n^{\alpha n-k+1}\int_t^\infty e^{(\alpha \ln y-y-\frac{k}{n}\ln y)n}dy.
\end{equation}
We let $g(y)=\alpha \ln y-y-\frac{k}{n}\ln y$ and thus $g'(y)=\frac{1}{y}\left(\alpha-\frac{k}{n}\right)-1$ implies $g$ is maximized at $\alpha-\frac{k}{n}$.  For $t<\alpha$ and $k$ (not dependent on $t$) there is a large enough $n$ so that $\alpha-\frac{k}{n}\in(t,\infty)$.  We may then use the Laplace method.
\begin{align*}
	\int_t^\infty e^{(\alpha\ln y-y-\frac{k}{n}\ln y)n}dy\sim
		&\int_{-\infty}^\infty e^{n\left(\left(\alpha-\frac{k}{n}\right)
		\left(\ln\left(\alpha-\frac{k}{n}\right)-1\right)
		-\frac{1}{2}\frac{1}{\alpha-\frac{k}{n}}
		\left(y-\alpha+\frac{k}{n}\right)^2\right)}dy
	\\=&e^{n\left(\left(\alpha-\frac{k}{n}\right)
		\left(\ln\left(\alpha-\frac{k}{n}\right)-1\right)\right)}
		\int_{-\infty}^\infty e^{\frac{-\left(y-\alpha+\frac{k}{n}\right)^2}
		{2\frac{\alpha n-k}{n^2}}}dy
	\\=&e^{n\left(\left(\alpha-\frac{k}{n}\right)
		\left(\ln\left(\alpha-\frac{k}{n}\right)-1\right)\right)}
		\frac{\sqrt{2\pi(\alpha n-k)}}{n}.
\end{align*}
Combining the above asymptotic with (\ref{incGamma}) and (\ref{GammaChange}), it follows from Stirling's formula that
\[\lim_{n\rightarrow\infty}\frac{a_{\alpha n-k}}{e^{nt}}=\lim_{n\rightarrow\infty}\frac{a_{n}}{e^{nt}}=1,\]
where we simply plug in $1$ for $\alpha$ and $0$ for $k$ to find $a_n$.  Then $b_{\alpha n}\sim b_n\sim nte^{nt}$ and $c_{\alpha n}\sim c_n\sim(nt+n^2t^2)e^{nt}$ and
\[\lim_{n\rightarrow\infty}\frac{r_1^2+r_2^2-2r_{12}^2}{n^{3/2}t r_3^2\sqrt{(r_1+r_2)^2-4r_{12}^2}}=\frac{1}{2}\sqrt{t},\quad\text{for $t<\alpha$}.\]

For $t>\alpha$ we will use the endpoint analysis of the Laplace method as outlined in ~\cite[\S3.3]{Miller}.  We begin by using the incomplete gamma function as before, defining $\kappa=\alpha+O(\frac{1}{n})$ so that $\kappa n\in\mathbb{N}$
\[a_{\kappa n}=\frac{e^{nt}n^{\kappa n+1}}{(\alpha n)!}\int_{t}^\infty e^{n(\kappa\log y-y)}dy.\]
Now set $g(y)=\kappa\log y-y$.  For large enough $n$, this function is maximized at $t$ for $t>\alpha$, which can be seen by $g'(y)=\frac{\kappa}{y}-1<0$ for large enough $n$.  Let $\tilde{g}(y)=g(y)-g(t)$.  Then $a_{\kappa n}$ can be rewritten as
\[a_{\kappa n}=\frac{n^{\kappa n+1}t^{\kappa n}}{(\kappa n)!}\int_{0}^\infty e^{n\tilde{g}(t+y)}dy.\]
Then there is a function, $y(s)$, guaranteed by the implicit function theorem, such that $\tilde{g}(t+y(s))=-s$.  We use this as a change of variables and obtain
\[a_{\kappa n}=\frac{n^{\kappa n+1}t^{\kappa n}}{(\kappa n)!}\int_{0}^\infty e^{-ns}y'(s)ds\]
which is in the form necessary to use Watson's lemma, which we rewrite to our specific context.
\begin{theorem}
	Suppose $y'(s)$ is smooth around a neighborhood of $s=0$ and absolutely integrable on $[0,\infty)$.  Then the 
	exponential integral
	\[F(n):=\int_0^\infty e^{-ns}y'(s)ds\]
	is finite for all $n>0$ and it has the asymptotic expansion
	\[F(n)\sim\sum_{k=0}^\infty\frac{y^{(k+1)}(0)}
	{n^{k+1}}\quad\text{as $n\rightarrow\infty$.}\]
\end{theorem}
One can solve for the derivatives of $y(s)$ by recursively using Taylor polynomials, though the first three in terms of $g(y)$ and its derivatives are computed in ~\cite{Miller}.
\[y'(0)=\frac{-1}{g'(t)}=-\frac{t}{\kappa-t},\qquad y''(0)=-\frac{g''(t)}{g'(t)^3}=\frac{\kappa t}{(\kappa-t)^3},\]
\[y'''(0)=\frac{g'''(t)g'(t)-3g''(t)^2}{g'(t)^5}=\frac{t}{(\kappa-t)^5}\left(2(\kappa-t)-3\kappa^2\right)\]
thus
\begin{equation}\label{Milleran}
	a_{\kappa n}=\frac{(nt)^{\kappa n}}{(\kappa n)!}
	\left(\frac{t}{t-\kappa}-\frac{\kappa t}{(t-\kappa)^3}\frac{1}{n}
	-\frac{t(2\kappa-3\kappa^2-2t)}{(t-\kappa)^5}\frac{1}{n^2}+O(n^{-3})\right)
\end{equation}
as $n\rightarrow\infty$.  It follows that
\begin{align}
	b_{\kappa n}=&\frac{(nt)^{\kappa n+1}}{(\kappa n)!}
		\left(\frac{\kappa}{t-\kappa}-\frac{\kappa t}{(t-\kappa)^3}\frac{1}{n}
		-\frac{t(2\kappa-3\kappa^2-2t)}{(t-\kappa)^5}\frac{1}{n^2}
		+O(n^{-3})\right),	
		\label{Millerbn}
	\\c_{\kappa n}=&\frac{(nt)^{\kappa n+2}}{(\kappa n)!}
		\left(\frac{\kappa^2}{t(t-\kappa)}+\frac{(\kappa-1)t^2-2t\kappa^2
		+\kappa^3}{t(t-\kappa)^3}\frac{1}{n}\right)
		\label{Millercn}
	\\ &+\frac{(nt)^{\kappa n+2}}{(\kappa n)!}
		\left(\frac{(2-\kappa)t^2+(5\kappa^2-2\kappa) t-\kappa^3}{(t-\kappa)^5}
		\frac{1}{n^2}+O(n^{-3})\right).\nonumber
\end{align}
We now notice that for $0<\alpha\le t<1$,
\[\lim_{n\rightarrow\infty}\frac{(nt)^{\kappa n}}{(\kappa n)!e^{nt}}=0.\]
Since $a_n\sim e^{nt}$ is still true when $\alpha\le t<1$ it follows that
\[\lim_{n\rightarrow\infty}\frac{r_3}{e^{nt}}=\lim_{n\rightarrow\infty}\frac{a_n}{e^{nt}}=1,\quad\lim_{n\rightarrow\infty}\frac{r_1}{nte^{2nt}}=1,\quad\text{and}\quad\lim_{n\rightarrow\infty}\frac{r_2}{nte^{2nt}}=\lim_{n\rightarrow\infty}\frac{r_{12}}{nte^{2nt}}=0.\]
Thus when $0<\alpha\le t<1$, we have
\[\lim_{n\rightarrow\infty}\frac{r_1^2+r_2^2-2r_{12}}{n^{3/2}tr_3^2\sqrt{(r_1+r_2)^2-4r_{12}^2}}=\lim_{n\rightarrow}\frac{1}{\sqrt{n}}=0.\]
Similarly, when $0<\alpha<1<t$,
\[\lim_{n\rightarrow\infty}\frac{n!(nt)^{\kappa n}}{(\kappa n)!(nt)^{n}}=0.\]
For this region, $a_{\kappa n}$, $b_{\kappa n}$, $c_{\kappa n}$, $a_n$, $b_n$, and $c_n$ are all calculated based on (\ref{Milleran}), (\ref{Millerbn}), and (\ref{Millercn}) where $\kappa=1$ for $a_n$, $b_n$ and $c_n$.  Therefore,
\[\lim_{n\rightarrow\infty}\frac{r_3}{a_n}=1,\quad\lim_{n\rightarrow\infty}\frac{r_1}{a_nc_n-b_n^2}=1,\quad\text{and}\quad\lim_{n\rightarrow\infty}\frac{r_2}{\frac{(nt)^{2n+2}}{(n!)^2}}=\lim_{n\rightarrow\infty}\frac{r_{12}}{\frac{(nt)^{2n+2}}{(n!)^2}}=0.\]
Thus when $0<\alpha<1<t$,
\begin{align*}
	\lim_{n\rightarrow\infty}\frac{r_1^2+r_2^2-2r_{12}}
	{n^{3/2}tr_3^2\sqrt{(r_1+r_2)^2-4r_{12}^2}}
	=&\lim_{n\rightarrow\infty}\frac{(a_nc_n-b_n^2)^2}
	{n^{3/2}ta_n^2(a_nc_n-b_n^2)}
	\\=&\lim_{n\rightarrow\infty}\frac{1}{n^{3/2}(t-1)^2}=0.
\end{align*}

In the case that $m$ is fixed, $a_n$, $b_n$, and $c_n$ again have asymptotic growth of $e^{nt}$, $nte^{nt}$, and $(n^2t^2+nt)e^{nt}$, respectively when $t<1$.  In the case that $t>1$, we have the same asymptotic growth given by (\ref{Milleran}), (\ref{Millerbn}), and (\ref{Millercn}), with $\alpha=1$.  Clearly,
\[\lim_{n\rightarrow\infty}\frac{k^{\beta}(nt)^k}{e^{nt}}=0\quad\text{so}\quad\lim_{n\rightarrow\infty}\frac{a_m}{e^{nt}}=\lim_{n\rightarrow\infty}\frac{b_m}{e^{nt}}=\lim_{n\rightarrow\infty}\frac{c_m}{e^{nt}}=0\]
for $0<t<1$.
Similarly, since for $t>1$,
$\displaystyle\lim_{n\rightarrow\infty}n!/(nt)^{n-j}=0$
then
\[\lim_{n\rightarrow\infty}\frac{a_mn!}{(nt)^n}=\lim_{n\rightarrow\infty}\frac{b_mn!}{(nt)^{n+1}}=\lim_{n\rightarrow\infty}\frac{c_mn!}{(nt)^{n+2}}=0.\]
If we compute the pointwise limit when $t<1$,
\[\frac{r_1^2+r_2^2-2r_{12}}{tr_3^2\sqrt{(r_1+r_2)^2-4r_{12}^2}}\cdot\frac{e^{-4nt}}{e^{-4nt}}\sim\frac{(nt+n^2t^2-n^2t^2)^2+0-0}{t\sqrt{(nt+n^2t^2-n^2t^2)^2-0}}=n.\]
Then, for $t>1$,
\[\frac{r_1^2+r_2^2-2r_{12}}{tr_3^2\sqrt{(r_1+r_2)^2-4r_{12}^2}}\cdot\frac{\frac{(n!)^4}{(nt)^{4n}}}{\frac{n!}{(nt)^n}}\sim\frac{(\frac{t^3}{(t-1)^4})^2+0-0}{t\left(\frac{t}{t-1}\right)^2\sqrt{(\frac{t^3}{(t-1)^4})^2-0}}=\frac{1}{(t-1)^2}.\]
Thus if we divide the integrand by $n$, as $n\rightarrow\infty$ we have a pointwise limit of
\begin{equation}
	\label{lim:fixed}
	\frac{r_1^2+r_2^2-2r_{12}}{ntr_3^2\sqrt{(r_1+r_2)^2-4r_{12}^2}}
	\rightarrow\left\{
	\begin{array}{ll}
		1&t<1\\0&t>1
	\end{array}\right._.
\end{equation}

\subsection{Bounding Functions}
In this section we obtain the bounding functions necessary to apply Theorem \ref{GLDCT}.  When $m=\alpha n$ for $\alpha\in(0,1)$, we divide \ref{MainInt} by $m^{3/2}=\alpha^{3/2}n^{3/2}$ in order to establish (\ref{MainThm}).
\begin{equation}\label{MainInt2}\frac{r_1^2+r_2^2-2r_{12}^2}{n^{3/2}tr_3^2\sqrt{(r_1+r_2)^2-4r_{12}^2}}.
\end{equation}
Then we have the following to help us simplify the denominator:
\begin{lemma}\label{CSApplication}
If $x$ is a positive real number and $a_k=\sum_{j=0}^k\frac{x^j}{j!}$, $b_k=\sum_{j=1}^k\frac{jx^j}{j!}$, and $c_k=\sum_{j=1}^k\frac{j^2x^j}{j!}$, then $(a_n+a_m)c_nc_m\ge c_nb_m^2+c_mb_n^2$, for $n,m\in\mathbb{N}$.
\end{lemma}
\noindent which can be proven using a couple applications of Cauchy-Schwarz.  We can find a lower bound on the quantity in the square root:
\begin{align*}
(r_1+r_2)^2-4r_{12}^2=&r_1^2+r_2^2+2r_1r_2-4r_{12}^2
\\=&r_1^2+r_2^2+2(a_n+a_m)^2c_nc_m-2(a_n+a_m)c_nb_m^2
\\&-2(a_n+a_m)c_mb_n^2-2r_{12}^2
\\\ge& r_1^2+r_2^2+2(a_n+a_m)^2(c_nc_m-c_nc_m)-2r_{12}^2
\\=&r_1^2+r_2^2-2r_{12}^2
\end{align*}
resulting in a bounding function for (\ref{MainInt2}) of
\begin{align*}\frac{\sqrt{r_1^2+r_2^2-2r_{12}^2}}{n^{3/2}tr_3^2}=&\frac{\sqrt{(r_1-r_2)^2+2r_3(r_1r_2-r_{12}^2)}}{n^{3/2}tr_3^2}
\\\le&\frac{r_1-r_2}{n^{3/2}tr_3^2}+\frac{\sqrt{2(r_1r_2-r_{12}^2)}}{n^{3/2}tr_3^{3/2}}.
\end{align*}
Then we note that
\[\frac{d}{dt}\left[a_k\right]=n\sum_{j=1}^kj\frac{(nt)^{j-1}}{j!}=\frac{1}{t}b_k,\text{ and }\frac{d}{dt}\left[b_k\right]=\frac{1}{t}c_k,\] and thus
\[\frac{d}{dt}\left[\frac{b_n-b_m}{a_n+a_m}\right]=\frac{(a_n+a_m)(c_n-c_m)-b_n^2+b_m^2}{t(a_n+a_m)^2}=\frac{r_1-r_2}{tr_3^2}.\]
Then
\[\lim_{t\to\infty}\int_\mathbb{R}\frac{r_1-r_2}{n^{3/2}tr_3^2}dt=\lim_{t\rightarrow\infty}\frac{\sum_{k=m+1}^nk\frac{(nt)^k}{k!}}{n^{3/2}(a_m+a_n)}=n^{-1/2}.\]
Moreover,
\[\frac{r_1r_2-r_{12}^2}{n^3t^2r_3^3}=\frac{(a_n+a_m)c_nc_m-c_nb_m^2-c_mb_n^2}{n^3t^2(a_n+a_m)^2}.\]
To establish a bound for this portion we will require the following lemma.
\begin{lemma}\label{ac-b2}
$a_kc_k-b_k^2\le a_{k-1}b_k$.
\end{lemma}
\begin{proof}
We have that $c_k=ntb_{k-1}+b_k$ and $nta_k=b_{k+1}$.  Then
\[a_kc_k=a_k(ntb_{k-1}+b_k)=b_{k+1}b_{k-1}+a_kb_k\]
so
\begin{align*}a_kc_k-b_k^2=b_{k+1}b_{k-1}-b_k^2+a_kb_k=&\frac{(nt)^{k+1}}{k!}b_{k-1}-\frac{(nt)^{k}}{(k-1)!}b_k+a_kb_k
\\=&\frac{(nt)^{k+1}}{k!}b_{k-1}-\frac{(nt)^{k}}{(k-1)!}b_k+\frac{(nt)^k}{k!}b_k+a_{k-1}b_k.
\end{align*}
Thus we prove our lemma if $\frac{(nt)^{k+1}}{k!}b_{k-1}-\frac{(nt)^{k}}{(k-1)!}b_k+\frac{(nt)^k}{k!}b_k\le0$.  We use the fact that $b_k=nta_{k-1}$ to say that
\[\frac{(nt)^k}{k!}b_k=\frac{(nt)^{k+1}}{k!}a_{k-1}.\]
Thus,
\begin{align*}\frac{(nt)^{k+1}}{k!}b_{k-1}-\frac{(nt)^{k}}{(k-1)!}b_k+\frac{(nt)^k}{k!}b_k=&\frac{(nt)^{k+1}}{k!}(b_{k-1}+a_{k-1})-\frac{(nt)^{k}}{(k-1)!}b_k
\\\le&\frac{(nt)^{k+1}}{k!}((k-1)a_{k-1}+a_{k-1})-\frac{(nt)^{k}}{(k-1)!}b_k
\\=&k\frac{(nt)^{k}}{k!}(nta_{k-1})-\frac{(nt)^{k}}{(k-1)!}b_k=0.
\end{align*}
\end{proof}

Equipped with this lemma and the fact that $c_k=ntb_{k-1}+b_k$ and $nta_k=b_{k+1}$ we can see
\begin{align*}r_1r_2-r_{12}^2\le &b_n^2a_nb_{m-1}+b_na_na_{n-1}b_m+b_n^2a_{m-1}a_{m-1}+b_na_{n-1}b_{m+1}b_{m-1}
\\&+b_{n+1}b_{n-1}b_ma_m+b_na_nb_ma_m+b_{n-1}b_m^2a_m+b_nb_ma_ma_{m-1}.
\end{align*}
Since $n\ge m$, it is clear each term in this sum is less than or equal to $b_n^2a_nb_{m}$ or $b_na_na_{n-1}b_{m}$ for sufficiently large $n$ and thus we search for an integrable bound on the root of these terms divided by $n^{3/2}tr_3^2$.
\[\frac{\sqrt{b_n^2a_nb_{m}}}{n^{3/2}tr_3^2}=\frac{a_{n-1}\sqrt{ta_na_{m-1}}}{(a_n+a_m)^2}\le\sqrt{\frac{ta_{m-1}}{a_n}}\le\left\{\begin{array}{ll}
\sqrt{t},& t<\delta\\t^{-(n-m)/2},&t>\delta
\end{array}\right.\]which is integrable for $\delta>1$ and $m\le n-4$.
\[\frac{\sqrt{b_na_na_{n-1}b_m}}{n^{3/2}tr_3^2}=\frac{a_{n-1}\sqrt{a_na_{m-1}}}{n^{1/2}(a_n+a_m)^2}\le \sqrt{\frac{a_{m-1}}{na_n}}\le\left\{\begin{array}{ll}
n^{-1/2},& t<\delta\\n^{-1/2}t^{-(n-m+1)/2},&t>\delta
\end{array}\right.\]which is again integrable for $\delta>1$ and $m\le n-3$.
Thus for $m=\alpha n$, $\alpha\in(0,1)$ we have the integrable sequence of bounding functions
\begin{equation}\label{alphanbound}
g_{n,m}(t)=\frac{r_1-r_2}{n^{3/2}tr_3^2}+16\left\{\begin{array}{ll}
\sqrt{t+1},&t<\delta\\t^{-(n-m)/2},&t>\delta
\end{array}\right._.
\end{equation}
Notice that we must exclude the case when $\alpha=1$ in order for $g_{n,\alpha n}$ to be integrable for some large $n$, thus we address this case separately.

When $m=n$, we may reduce the integrand in (\ref{MainInt}) to
\begin{equation}\label{eqnMequalsN}
\frac{\sqrt{a_nc_n(a_nc_n-b_n^2)}}{2n^{3/2}ta_n^2}
\end{equation}
We notice
\begin{align*}
\frac{a_nc_n(a_nc_n-b_n^2)}{n^3t^2a_n^4}=\frac{c_n}{n^2ta_n}\frac{a_nc_n-b_n^2}{nta_n^2}\le\frac{c_n}{n^2ta_n}\frac{a_{n-1}b_n}{nta_n^2}\le\frac{(n^2t^2+nt)a_n}{n^2ta_n}\frac{nta_n}{nta_n^2}=t+\frac{1}{n}
\end{align*}
giving us an upper bound of $\frac{1}{2}\sqrt{t+1}$ for (\ref{eqnMequalsN}), where the first inequality follows from Lemma \ref{ac-b2}.
For the tail we use the fact that $a_n\ge ta_{n-1}$ along with $c_n\le n^2a_n$ and Lemma \ref{ac-b2} to see that
\[\frac{\sqrt{a_nc_n(a_nc_n-b_n^2)}}{n^{3/2}ta_n^2}=\sqrt{\frac{c_n}{n^2a_n}}\sqrt{\frac{a_nc_n-b_n^2}{nt^2a_n^2}}\le\sqrt{\frac{a_{n-1}b_n}{nt^2a_n^2}}=\sqrt{\frac{a_{n-1}^2}{ta_n^2}}\le t^{-3/2}\]
which gives us a bounding function for $t\in(0,\infty)$ of
\[g(t)=\frac{1}{2}\cdot\left\{\begin{array}{ll}
\textstyle\sqrt{t+1},&t<\delta\\t^{-3/2},&t>\delta
\end{array}\right.\]

For $m$ fixed we need to show asymptotic growth of $\mathbb{E}\mathcal{N}$ to be $n$ and thus we divide by this before we take the limit.  We again use lemma \ref{CSApplication} and the subsequent simplification used in the $m=\alpha n$ case to bound our integrand by
\[\frac{\sqrt{r_1^2+r_2^2-2r_{12}^2}}{ntr_3^2}=\frac{\sqrt{(r_1-r_2)^2+2(r_1r_2-r_{12}^2)}}{ntr_3^2}\le\frac{r_1-r_2}{ntr_3^2}+\sqrt{2}\frac{\sqrt{r_1r_2-r_{12}^2}}{ntr_3^2}.\]
Then we have
\[\sqrt{2\frac{r_1r_2-r_{12}^2}{n^2t^2r_3^4}}=\sqrt{2\frac{r_3c_nc_m-c_nb_m^2-c_mb_n^2}{n^2t^2r_3^3}}\le\frac{2c_m}{ntr_3}+\frac{r_3c_n-b_n^2-\frac{c_nb_m^2}{c_m}}{ntr_3^2}\]by comparison of arithmetic and geometric means.  By manipulating sums, we can easily see that $c_mb_n\le c_nb_m$ so
\[\sqrt{2\frac{r_1r_2-r_{12}^2}{n^2t^2r_3^4}}\le\frac{2c_m}{ntr_3}+\frac{r_1-r_{12}}{ntr_3^2}.\]
Now, we need to bound $c_m/ntr_3$ by an integrable function.
\[\frac{c_m}{ntr_3}\le\frac{m^2a_m}{nta_{m+1}}\le\frac{m^3\frac{(nt)^{k-1}}{k!}}{\frac{(nt)^{k-1}}{(k-1)!}+\frac{(nt)^{k+1}}{(k+1)!}}\le\frac{m^3(m+1)}{1+(nt)^{2}}\]which is integrable on $(0,\infty)$ with an integral of $\frac{m^3(m+1)\pi}{2n}$.
Furthermore,
\[\frac{d}{dt}\left[\frac{2b_n-b_m}{a_n+a_m}\right]=\frac{2r_1-r_2-r_{12}}{tr_3^2}\]which is the remaining portion of our integrand bound.  Since
\[\lim_{t\rightarrow\infty}\frac{1}{n}\frac{2b_n-b_m}{a_n+a_m}=2,\] this is also integrable on $t\in(0,\infty)$ giving us an integrable bounding function of
\[\frac{r_1^2+r_2^2-2r_{12}^2}{ntr_3^2\sqrt{(r_1+r_2)^2-4r_{12}^2}}\le \frac{2r_1-r_2-r_{12}}{ntr_3^2}+\frac{m^3(m+1)}{1+(nt)^{2}}.\]

\subsection{Applying Theorem \ref{GLDCT}}
In all cases we now have the tools to apply the Generalized Dominated Convergence Theorem.  For $m=\alpha n$ we have the sequence of bounding functions given in (\ref{alphanbound})
whose limit integral is
\begin{equation}
\lim_{n\rightarrow\infty}\int_\mathbb{R}g_{n,\alpha n}(t)\,dt=\lim_{n\rightarrow\infty}n^{-1/2}+{\textstyle\frac{16}{3}}(\delta+1)^{3/2}+\frac{8\delta^{-((1-\alpha)n-2)/2}}{n(1-\alpha)-2}={\textstyle\frac{16}{3}}(\delta+1)^{3/2}
\end{equation}if $n>2/(1-\alpha)$,
and whose limit is
\begin{equation}
g(t)=\left\{\begin{array}{ll}8\sqrt{t+1},&t<\delta\\0,&t>\delta\end{array}\right.
.\end{equation}
Thus we have the necessary condition
\begin{equation}
\lim_{n\rightarrow\infty}\int_\mathbb{R}g_{n,\alpha n}(t)\,dt=\int_\mathbb{R}g(t)\,dt=\frac{16}{3}(\delta+1)^{3/2}
\end{equation}
and we may apply the theorem to say
\begin{equation}
\lim_{n\rightarrow\infty}\frac{1}{n^{3/2}}\mathbb{E}\mathcal{N}_H(\mathbb{C})=\int_0^\alpha\frac{1}{2}\sqrt{t}dt=\frac{1}{3}\alpha^{3/2}
\end{equation}
implying $\mathbb{E}\mathcal{N}_H(\mathbb{C})\sim \frac{1}{3}m^{3/2}$.

For $m=n$, the result is the same.  This can be see by noting
\[\int_\mathbb{R}g(t)=\lim_{n\rightarrow\infty}\frac{1}{3}((\delta+1)^{3/2}-1)+\frac{1}{4}\delta^{-1/2}\] and using the Lebesgue Dominated Convergence Theorem.

For the $m$ fixed case, we have the sequence of bounding functions,
\begin{equation}
g_n=\frac{2r_1-r_2-r_{12}}{ntr_3^2}+\frac{m^3(m+1)}{1+(nt)^2}
\end{equation}
and as we mentioned above
\begin{equation}
\int_0^\infty\lim_{n\rightarrow\infty}g_n(t)dt=\int_0^12dt=2=\lim_{n\rightarrow\infty}2+\frac{m^3(m+1)\pi}{2n}=\lim_{n\rightarrow\infty}\int_0^\infty g_n.
\end{equation}
Applying theorem \ref{GLDCT}, we get
\begin{equation}
\lim_{n\rightarrow\infty}\frac{1}{n}\mathbb{E}\mathcal{N}_H(\mathbb{C})=\int_0^1dt=1\quad\implies \quad\mathbb{E}\mathcal{N}_H(\mathbb{C})\sim n.
\end{equation}

\section{First Intensity Functions}\label{sec1stInt}
Inspecting the first intensity functions offers some insight into the distribution of zeros of the Weyl polynomials.

Define the first intensity function for the expected number of zeros of $H_{n,m}(z)$ over a domain $T$,
\[I_{n,m}(z)=\frac{r_1^2+r_2^2-2r_{12}^2}{\lvert z\rvert^2r_3^2\sqrt{(r_1+r_2)^2-4r_{12}^2}}\]then for the analytic case we have
\[I_{n,0}(z)=\frac{r_1}{\lvert z\rvert^2r_3^2}.\]
Interestingly $I_{n,m} \approx \chi_{|z|\leq 1}$  (in fact by (\ref{lim:fixed}) we have equality as $n \rightarrow \infty$).
\begin{figure}[ht]
\begin{subfigure}[b]{.25\textwidth}
\includegraphics[width=\textwidth]{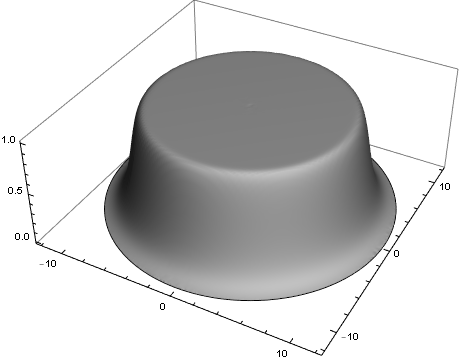}
\end{subfigure}\quad
\begin{subfigure}[b]{.3\textwidth}
\includegraphics[width=\textwidth]{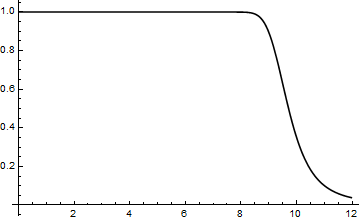}
\end{subfigure}
\caption{Plot of $I_{100,0}(z)$ in 3 dimensions $(x,y,I(z))$, $z=x+iy$.}
\end{figure}
This is quite different than the traditional Kac model in which all the zeros accumulate on the circle of radius $1$.  

$I_{n,0}$ is of particular interest since $I_{n,m=m_0}\rightarrow I_{n,0}$ as $n\rightarrow\infty$ and, perhaps more importantly, it offers a clue into how $p$ and $q$ contribute to the zeros of $H$.
Consider the difference of the $I_{100,16}$ and $I_{100,0}$. 
\begin{figure}[ht]\label{fig:alphan}
\begin{subfigure}[b]{.55\textwidth}\label{alphan1}
\includegraphics[width=.45\textwidth]{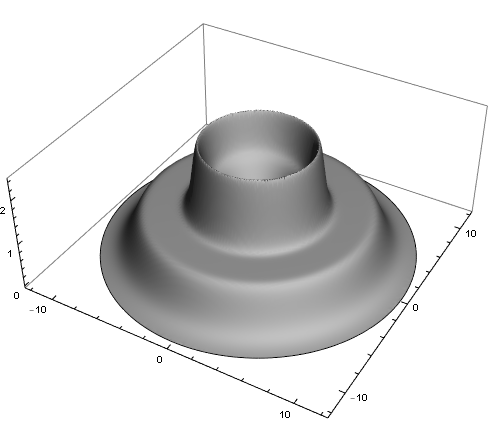}
\quad
\includegraphics[width=.45\textwidth]{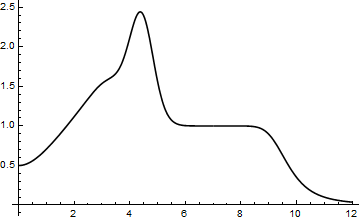}
\caption{$I_{100,16}(z)$}
\end{subfigure}
\hspace{-.04\textwidth}
\begin{subfigure}[b]{.27\textwidth}\label{alphan2}
\includegraphics[width=\textwidth]{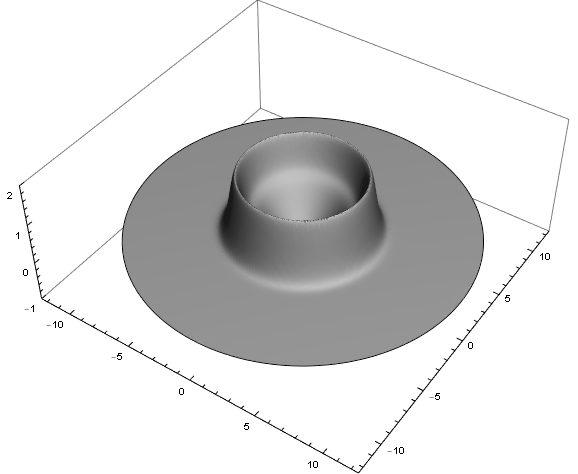}
\caption{$I_{100,16}-I_{100,0}$}
\end{subfigure}
\caption{}
\end{figure}
It appears that $q_m$ only contributes to the zeros of $H$ which lie inside the disc of radius $\sqrt{m}$.

\section{Conclusion}
We have shown that in the case that $m=\alpha n+O(1)$ that the expected number of zeros of a random Weyl polynomial has growth-order $\frac{1}{3}m^{\frac{3}{2}}$ Moreover, in this paper as well as in Li and Wei's work on the Kostlan model [4], Lerario and Lundberg's work on the truncated model [3] and Thomack's work on the naive model [5] we have the result that when $m$ is fixed the number of zeros grows linearly in $n$. This raises the question: Is there a class of Gaussian harmonic polynomials such that the expected number of zeros in the $m$ fixed case increases faster than $n$?

We end with a conjecture motivated by the suggestive material in \S\ref{sec1stInt}.
\begin{conjecture}
For $H_{n,m}(z)$ a random harmonic polynomial following the Weyl model,
\[\mathbb{E}\mathcal{N}_{H}(\mathbb{C})\sim \textstyle\frac{1}{3}m^{3/2}+n+O(\sqrt{n}).\]
\end{conjecture}

\bibliographystyle{amsplainABBREV}

\bibliography{references}

\end{document}